\documentclass[12pt,reqno]{amsart}

\usepackage[usenames, dvipsnames, svgnames]{xcolor}
\usepackage{amsmath, amssymb,graphicx,amsthm,latexsym, amsfonts, enumitem, mathtools, tensor, MnSymbol}
\usepackage{hyperref}
\usepackage[all, color]{xy}
\usepackage{color}
\usepackage{amssymb}
\usepackage{float}
\usepackage{tikz}
\usetikzlibrary{arrows,decorations.pathmorphing,decorations.pathreplacing,positioning,shapes.geometric,shapes.misc,decorations.markings,decorations.fractals,calc,patterns}

\makeatletter
\@namedef{subjclassname@2020}{%
  \textup{2020} Mathematics Subject Classification}
\makeatother

\DeclareMathOperator{\Hom}{Hom}

\theoremstyle{plain}
\newtheorem{theorem}{Theorem}[section]
\newtheorem*{theorem*}{Theorem}

\theoremstyle{definition}
\newtheorem{defn}[theorem]{Definition}

\newtheorem{remark}[theorem]{Remark}

\newtheorem{lemma}[theorem]{Lemma}
\newtheorem{notation}[theorem]{Notation}
\newtheorem{corollary}[theorem]{Corollary}

\newtheorem{setup}[theorem]{Setup}

\date{}
\begin{document}
\setlength{\parindent}{0pt}
\setlength{\parskip}{7pt}
\title{Negative cluster categories from simple minded collection quadruples}
\author{Francesca Fedele}
\address{Dipartimento di Informatica - Settore di Matematica Universita` degli Studi di Verona
Strada le Grazie 15 - Ca` Vignal, I-37134 Verona, Italy}
\email{francesca.fedele@univr.it}
\begin{abstract}
Fomin and Zelevinsky's definition of cluster algebras laid the foundation for cluster theory.  The various categorifications and generalisations of the original definition led to Iyama and Yoshino's generalised cluster categories $\mathcal{T}/\mathcal{T}^{fd}$ coming from positive-Calabi-Yau triples $(\mathcal{T}, \mathcal{T}^{fd},\mathcal{M})$. Jin later defined simple minded collection quadruples $(\mathcal{T}, \mathcal{T}^{p},\mathbb{S},\mathcal{S})$, where the special case $\mathbb{S}=\Sigma^{-d}$ is the analogue of Iyama and Yang's triples: negative-Calabi-Yau triples.

In this paper, we further study the quotient categories $\mathcal{T}/\mathcal{T}^p$ coming from simple minded collection quadruples. Our main result uses limits and colimits to describe Hom-spaces over $\mathcal{T}/\mathcal{T}^p$ in relation to the easier to understand Hom-spaces over $\mathcal{T}$. Moreover, we apply our theorem to give a different proof of a result by Jin: if we have a negative-Calabi-Yau triple, then $\mathcal{T}/\mathcal{T}^p$ is a negative cluster category. 
\end{abstract}
\maketitle
\section{Introduction}
Cluster theory developed starting from the definition of cluster algebras by Fomin and Zelevinsky, see \cite[Definition 2.3]{FZ}.  Buan, Marsh, Reineke,
Reiten, and Todorov later categorified the original definition by introducing cluster categories in \cite[Section 1]{BRMRT}. Further categorifications have been widely studied in the past few years, see for example \cite{Gin}, \cite{AC}, \cite{GL}.
This led to two parallel generalisations of cluster categories coming from positive and negative Calabi-Yau triples respectively.

First, Iyama and Yang introduced generalised cluster categories, see \cite{IYa1}, which instead of coming from dg algebras, come from $d$-Calabi-Yau triples of the form $(\mathcal{T},\mathcal{T}^{fd},\mathcal{M})$ for some positive integer. The idea behind such a triple is that, starting from a triangulated category $\mathcal{T}$ with some extra assumptions and a triangulated subcategory $\mathcal{T}^{fd}$, we obtain a triangulated quotient category $\mathcal{T}/\mathcal{T}^{fd}$ which is a generalised cluster category: it is Hom-finite, $(d-1)$-Calabi-Yau and it has a $(d-1)$-cluster tilting object. This was first proven by Iyama and Yang in \cite[Section 5]{IYa1} and we presented an alternative proof using more classic means such as limits, colimits and a Gap Theorem in \cite{FF}.

Note that in this setup $(\mathcal{T}, \mathcal{T}^{fd})$ is relative $d$-Calabi-Yau in the sense that for every $X\in\mathcal{T}^{fd}$ and $Y\in\mathcal{T}$ there exists a bifunctorial isomorphism of the form
\begin{align}\label{diagram_CY_Serre}\tag{$\star$}
D\mathcal{T}(X,Y)\cong \mathcal{T}(Y,\Sigma^d X).
\end{align}

An interesting question is what happens if we subsitute $\Sigma^d$ in (\ref{diagram_CY_Serre}) with a more general ``restricted Serre functor'' $\mathbb{S}:\mathcal{T}\rightarrow\mathcal{T}$. In particular, in the same fashion negative-Calabi-Yau categories can be studied, we can look at the special case when we have $\Sigma^{-d}$ instead of $\Sigma^d$ in (\ref{diagram_CY_Serre}). This motivates the parallel generalisation of cluster categories to negative cluster categories, see \cite{CS2,CPP,IJ}.

Coelho Sim\~oes showed there is a parallel between the mutation theory for $d$-simple minded systems in  $(-d)$-Calabi-Yau triangulated categories and $d$-cluster-tilting objects in $d$-Calabi-Yau triangulated categories, see \cite[Theorems A-D and the subsequent paragraph]{CS}. 
Moreover, in \cite{JH}, Jin defined SMC quadruples $(\mathcal{T},\mathcal{T}^p,\mathbb{S},\mathcal{S})$, where the case $\mathbb{S}=\Sigma^{-d}$ is the analogue of Iyama and Yang's triples, that is $(-d)$-Calabi-Yau triples, see Definition \ref{defn_SMC_quadruple}. Moreover, Coelho Sim\~oes, Pauksztello and Ploog then introduced negative cluster categories in \cite{CPP}, which is the analogue of the ``classical'' cluster category considered by Buan, Marsh, Reineke,
Reiten, and Todorov in \cite{BRMRT} and by Thomas in \cite{TH}. 

The aim of this paper is to prove that if $(\mathcal{T},\mathcal{T}^p,\mathcal{S})$ is a $(-d)$-Calabi-Yau triple, then $\mathcal{T}/\mathcal{T}^p$ is a negative cluster category, that is it is Hom-finite, $(-d-1)$-Calabi-Yau and it has a $(d+1)$-simple minded system in the sense of Definition \ref{defn_SMS}. Note that Jin already proved this in \cite[Theorem 4.5]{JH}.
In the same way we reproved the corresponding result in the positive-Calabi-Yau setup in \cite{FF}, we present an alternative proof of this result which uses different and more classic means, giving in this way a deeper understanding of $\mathcal{T}/\mathcal{T}^p$. In particular, we use limits and colimits to explicitly describe Hom spaces in this quotient category.

Let $k$ be an algebraically closed field and $(\mathcal{T}, \mathcal{T}^p,\mathbb{S},\mathcal{S})$ be an SMC quadruple. By property (RS2) of SMC quadruples, see Definition \ref{defn_SMC_quadruple}, for every integer $i$, the pair
\begin{align*}
(\mathcal{T}_{>i},\mathcal{T}_{\leq i}):=({}^{\perp}(\Sigma^{>-i-1}\mathcal{S}), {}^{\perp}(\Sigma^{<-i}\mathcal{S})) 
\end{align*}
is a co-$t$-structure on $\mathcal{T}$. Hence, for every object $X$ in $\mathcal{T}$, there is a truncation triangle of the form
\begin{align*}
    X_{>i}\xrightarrow{f_i} X\xrightarrow{g_i} X_{\leq i}\rightarrow \Sigma X_{>i},
\end{align*}
where $X_{>i}\in\mathcal{T}_{>i}$ and $X_{\leq i}\in \mathcal{T}_{\leq i}$. Note that such a triangle is not unique and when we write $X_{> i}$, respectively $X_{\leq i}$, we mean an object fitting in one such triangle.

The main result of this paper gives a relation between Hom-spaces in $\mathcal{T}$ and Hom-spaces in $\mathcal{T}/\mathcal{T}^p$ using limits and colimits. Note that our theorem recalls some results by Artin and Zhang, see 	\cite[Propositions 2.2 and 3.13]{AZ}.

{\bf Theorem A (=Theorem \ref{thm_directstab}).}
{\em Let $X$ and $Y$ be objects in $\mathcal{T}$.
\begin{enumerate}[label=(\alph*)]
\item For $p\gg0$, each direct system
\begin{align*}
\xymatrix@C=3em{
    \mathcal{T}(Y,X)\ar[r]&\mathcal{T}(Y,X_{\leq 0})\ar[r]&\mathcal{T}(Y,X_{\leq -1})\ar[r]&\cdots\ar[r]&\mathcal{T}(Y,X_{\leq -p})\ar[r]&\cdots
    }
\end{align*}
stabilizes. Moreover, we have that
    $\underset{p}{\varinjlim}\,\mathcal{T}(Y,X_{\leq -p})\cong \mathcal{T}/\mathcal{T}^{p}(Y, X)$.
\item Fixing truncation triangles for X and Y, we have that
    \begin{align*}
       \underset{q}{\varprojlim} \Big(\underset{p}{\varinjlim}\,\mathcal{T}(Y_{\leq -q},\Sigma X_{> -p})\Big)\cong \mathcal{T}/\mathcal{T}^{p}(Y,X).
    \end{align*}
\end{enumerate}}

Assume now that $\mathbb{S}=\Sigma^{-d}$ for some positive integer $d$. Let $Y$ be in $\mathcal{S}$. Considering $\Sigma^j Y$ for any integer $j$, we prove that the system from part (a) of the theorem is stable for $p\geq j-d+1$, and so 
\begin{align*}
\mathcal{T}/\mathcal{T}^p(\Sigma^j Y, X)\cong \mathcal{T}(\Sigma^j Y, X_{\leq -j+d-1}),
\end{align*}
see Corollary \ref{coro_YisS}. If we further have $X\in\mathcal{S}$, we can say when the Hom-space over $\mathcal{T}/\mathcal{T}^p$ vanishes, see Corollary \ref{coro_prop(2)}.

We apply our theorem to prove the following result saying that $\mathcal{T}/\mathcal{T}^p$ is a negative cluster category coming from a negative-Calabi-Yau triple. Note that this is the analogue of Iyama and Yang's generalised cluster category coming from a positive-Calabi-Yau triple, see \cite[Section 5]{IYa1}.

{\bf Theorem B.}
{\em 
\begin{enumerate}[label=(\alph*)]
\item The category $\mathcal{T}/\mathcal{T}^p$ is Hom-finite and $(-d-1)$-Calabi-Yau.
\item The subcategory $\mathcal{S}\subseteq \mathcal{T}/\mathcal{T}^p$ is a $(d+1)$-simple minded system ($(d+1)$-SMS) in the sense of Definition \ref{defn_SMS}.
\end{enumerate}
}

Parts (a) and (b) correspond to Theorems \ref{lemma_negCY} and \ref{coro_SMS} respectively. Note that these results have  already been proven by Jin in \cite[Theorem 4.5]{JH} but we present a more efficient proof. We use Theorem A to prove both part (a) and properties (1) and (2) of Definition \ref{defn_SMS} and a lemma by Iyama and Yang from \cite{IYa} to prove property (3). Note that Jin relies on \cite{AC} to prove part (a) and on \cite{IYa} for his proof of part (b) but he applies the main theorem of Iyama and Yang's paper.

The paper is organised as follows. Section \ref{section_background} presents some background material and our setup. Section \ref{section_main} studies Hom spaces in $\mathcal{T}/\mathcal{T}^p$ and proves Theorem A. Finally, Section \ref{section_application} applies Theorem A to prove Theorem B.

\section{Background}\label{section_background}
The definition of simple-minded collections was first introduced by Koenig and Yang in \cite[Definition 3.2]{KY}, but the same concept was also defined earlier by Al-Nofayee in \cite{AN} as \textit{cohomologically Schurian set of generators}. We recall it here in the form presented in \cite[Definition 2.4]{JH} by Jin. Note that the base field $k$ is assumed to be algebraically closed throughout the paper.

\begin{defn}
Let $\mathcal{T}$ be a $k$-linear, Hom-finite, Krull-Schmidt triangulated category and $\mathcal{S}$ be a subcategory of $\mathcal{T}$. We say that $\mathcal{S}$ is a \textit{simple minded collection (SMC)} if the following hold for any objects $X$ and $Y$ in $\mathcal{S}$:
\begin{enumerate}
\item $\mathcal{T}(X, \Sigma^{<0}Y)=0$;
\item dim$_k \mathcal{T}(X,Y)=\delta_{XY}$;
\item thick$(\mathcal{S})=\mathcal{T}$.
\end{enumerate}
\end{defn}

\begin{notation}
Let $\mathcal{T}$ be a triangulated category and $\mathcal{X}$, $\mathcal{Y}$ two full subcategories of $\mathcal{T}$. We write
\begin{align*}
\mathcal{X}*\mathcal{Y}:=\{ T\in\mathcal{T}\mid \text{ there exists a triangle } X\rightarrow T\rightarrow Y\rightarrow \Sigma X \text{ in }\mathcal{T}, \text{ with } X\in\mathcal{X},\, Y\in\mathcal{Y} \}.
\end{align*}
When $\mathcal{T}(\mathcal{X},\mathcal{Y})=0$ holds, we write $\mathcal{X}*\mathcal{Y}=\mathcal{X}\perp\mathcal{Y}$. Moreover, we write
\begin{align*}
\mathcal{X}^\perp=\{T\in\mathcal{T}\mid \mathcal{T}(\mathcal{X},T)=0\} \,\,\text{ and }\,\,
{}^\perp\mathcal{X}=\{T\in\mathcal{T}\mid \mathcal{T}(T,\mathcal{X})=0\}.
\end{align*}
\end{notation}

\begin{defn}[{\cite[Definition 4.1]{JH}}]\label{defn_SMC_quadruple}
We say that a quadruple $(\mathcal{T}, \mathcal{T}^p,\mathbb{S},\mathcal{S})$ is a \textit{simple minded collection (SMC) quadruple} if the following are satisfied:
\begin{enumerate}
    \item[(RS0)] $\mathcal{T}$ is a $k$-linear, Hom-finite, Krull-Schmidt triangulated category and $\mathcal{T}^p$ is a thick subcategory of $\mathcal{T}$;
    \item[(RS1)] $\mathbb{S}:\mathcal{T}\rightarrow \mathcal{T}$ is a triangle equivalence restricting to an equivalence $\mathbb{S}:\mathcal{T}^p\rightarrow\mathcal{T}^p$ and satisfying a bifunctorial isomorphism for any $X\in\mathcal{T}^p$ and $Y\in\mathcal{T}$:
    \begin{align*}
        D\mathcal{T}(X,Y)\cong \mathcal{T}(Y,\mathbb{S}X);
    \end{align*}
    \item[(RS2)] $\mathcal{S}$ is a SMC in $\mathcal{T}$ and $\mathcal{T}= {}^{\perp} (\Sigma^{\geq 0}\mathcal{S})\perp{}^{\perp} (\Sigma^{< 0}\mathcal{S})= (\Sigma^{\geq 0}\mathcal{S})^{\perp}\perp (\Sigma^{< 0}\mathcal{S})^{\perp}$ are co-$t$-structures of $\mathcal{T}$ satisfying ${}^{\perp} (\Sigma^{\geq 0}\mathcal{S})\subset \mathcal{T}^p$ and $(\Sigma^{< 0}\mathcal{S})^{\perp}\subset \mathcal{T}^p$.
\end{enumerate}
Moreover, if $\mathbb{S}=\Sigma^{-d}$ for some integer $d\geq 0$, we say that $(\mathcal{T}, \mathcal{T}^p, \mathcal{S})$ is a \textit{$(-d)$-Calabi-Yau ($(-d)$-CY) triple}.
\end{defn}

The definition of $n$-simple minded system was first introduced in \cite[Section 1]{CS} and, applying \cite[Lemma 2.8]{CP} to the equivalent definition \cite[Definition 2.1]{CP}, one can see it is equivalent to the following.

\begin{defn}\label{defn_SMS}
Let $\mathcal{T}$ be a $k$-linear, Hom-finite, Krull-Schmidt triangulated category and $n\geq 1$ be an integer. A collection of objects $\mathcal{C}$ in $\mathcal{T}$ is an \textit{$n$-simple minded system ($n$-SMS)} if:
\begin{enumerate}
\item dim$_k \mathcal{T}(X,Y)=\delta_{XY}$ for every $X,\,Y$ in $\mathcal{C}$;
\item if $n\geq 2$, $\Hom(\Sigma^m X,Y)=0$ for $1\leq m\leq n-1$ and $X,\,Y$ in $\mathcal{C}$;
\item $\mathcal{T}=\langle \mathcal{C}\rangle * \Sigma^{-1}\langle\mathcal{C}\rangle*\dots* \Sigma^{1-n}\langle\mathcal{C}\rangle$.
\end{enumerate}
\end{defn}
\begin{remark}
Note that we are following the definition as stated in \cite{CP}, while Jin applies a shift to it when defining an $n$-SMS in \cite[Definition 2.7]{JH}.
\end{remark}

\begin{setup}
Let $(\mathcal{T}, \mathcal{T}^p,\mathbb{S},\mathcal{S})$ be an SMC quadruple.
\end{setup}

\begin{notation}
Let $i$ be an integer, we set $\mathcal{T}_{>i}=\mathcal{T}_{\geq i+1}:={}^{\perp}(\Sigma^{>-i-1}\mathcal{S})$, $\mathcal{T}_{< i+1}=\mathcal{T}_{\leq i}:={}^{\perp}(\Sigma^{<-i}\mathcal{S})$, $\mathcal{T}^{> i}= \mathcal{T}^{\geq i+1}:=\langle \Sigma^{<-i}\mathcal{S}\rangle$ and $\mathcal{T}^{< i}= \mathcal{T}^{\leq i-1}:=\langle \Sigma^{\geq 1-i}\mathcal{S}\rangle$. 
\end{notation}
\begin{remark}\label{remark_cot_structure}
Note that $\mathcal{T}=\mathcal{T}_{>i}\perp \mathcal{T}_{\leq i}$ is a co-$t$-structure with $\mathcal{T}_{>i}\subset \mathcal{T}^p$. Hence $\mathcal{T}_{>i+1}\subseteq \Sigma \mathcal{T}_{>i+1}=\mathcal{T}_{>i}$ and $\mathcal{T}_{\leq i}\subseteq \Sigma^{-1}\mathcal{T}_{\leq i}=\mathcal{T}_{\leq i+1}$.
As $(\mathcal{T}_{>i},\mathcal{T}_{\leq i})$ is a torsion pair, for any $X\in\mathcal{T}$ there is a (non-unique) truncation triangle of the form
\begin{align*}
    X_{>i}\xrightarrow{f_i} X\xrightarrow{g_i} X_{\leq i}\rightarrow \Sigma X_{>i},
\end{align*}
where $X_{>i}\in\mathcal{T}_{>i}$, $X_{\leq i}\in \mathcal{T}_{\leq i}$ and we may assume that $g_i$ is left minimal.
When we write $X_{> i}$, respectively $X_{\leq i}$, we mean an object fitting in one of such truncation triangles.

Note also that $\mathcal{T}= \mathcal{T}^{\leq i}\perp \mathcal{T}^{>i}$ is a bounded $t$-structure and $\mathcal{T}^{\leq i}=\mathcal{T}_{\leq i}$, see \cite[Top of page 12]{JH}.
\end{remark}

\begin{lemma}\label{lemma_directsystem}
For any $X\in\mathcal{T}$, fixing a set of truncation triangles as described in Remark \ref{remark_cot_structure}, there is a direct system of the form
\begin{align*}
    X\xrightarrow{g_0} X_{\leq 0}\xrightarrow{\xi_0} X_{\leq -1}\xrightarrow{\xi_{-1}} X_{\leq -2}\xrightarrow{\xi_{-2}}\cdots.
\end{align*}
\end{lemma}
\begin{proof}
Given $X$ in $\mathcal{T}$ and an integer $i$, then fixing truncation triangles as in Remark \ref{remark_cot_structure}, there is a morphism of triangles of the form:
\begin{align*}
    \xymatrix@C=3em{
    X_{> i}\ar[r]^{f_{i}}\ar[d]^{\alpha_{i}}& X\ar[r]^{g_i}\ar@{=}[d]& X_{\leq i}\ar[r]\ar@{-->}[d]^{\xi_i}&\Sigma X_{> i}\ar[d]\\
    X_{> i-1}\ar[r]_{f_{i-1}}& X\ar[r]_{g_{i-1}}& X_{\leq i-1}\ar[r]&\Sigma X_{> i-1},
    }
\end{align*}
where $\alpha_{i}$ such that $f_{i-1}\circ \alpha_{i}=f_{i}$ exists since $X_{> i}\in\mathcal{T}_{>i}\subset \mathcal{T}_{>i-1}$, $X_{\leq i-1}\in\mathcal{T}_{\geq i-1}$ and $(\mathcal{T}_{>i-1},\mathcal{T}_{\leq i-1})$ is a torsion pair and $\xi_i$ then exists by the axioms of triangulated categories. Then
\begin{align*}
    X\xrightarrow{g_0} X_{\leq 0}\xrightarrow{\xi_0} X_{\leq -1}\xrightarrow{\xi_{-1}} X_{\leq -2}\xrightarrow{\xi_{-2}}\cdots
\end{align*}
is a direct system.
\end{proof}

\section{Morphisms in $\mathcal{T}/\mathcal{T}^p$}\label{section_main}
The goal of this section is to prove Theorem \ref{thm_directstab}. The whole section closely follows \cite[Section 3]{FF} which proves the corresponding result for cluster categories coming from positive-Calabi-Yau triples. Most of the proofs in this section are the same as the proofs of the corresponding results in \cite{FF} up to using the direct system obtained from co-$t$-structure truncations and the Hom-vanishing coming from the simple-minded analogues in \cite{JH} of the corresponding silting results in \cite{IYa1}. Note that in \cite{FF} the truncation triangles come from $t$-structures and are hence unique. On the other hand, here they come from co-$t$-structures and are not unique, and so the direct system and the objects $C^X_{-p}$ introduced later in the section depend on the choice of the truncation triangles.

\begin{lemma}\label{lemma_iso_quotient_cat}
Let $X\in\mathcal{T}$. For any integer $m$, in the quotient category $\mathcal{T}/\mathcal{T}^{p}$ for each truncation triangle we have that $g_m$ and $\xi_m$ become  isomorphisms. Moreover, $X\cong_{\mathcal{T}/\mathcal{T}^{p}} X_{\leq m}$.
\end{lemma}
\begin{proof}
By (RS2) and since  $\mathcal{T}^{p}$ is closed under integer powers of $\Sigma$, we have that $\mathcal{T}_{> m}\subset \mathcal{T}^{p}$ for any integer $m$. Then any truncation triangle
\begin{align*}
    X_{> m}\xrightarrow{f_m} X\xrightarrow{g_{m}} X_{\leq m}\rightarrow\Sigma X_{> m}
\end{align*}
viewed as a triangle in $\mathcal{T}/\mathcal{T}^{p}$ is such that $X_{> m}\cong_{\mathcal{T}/\mathcal{T}^{p}} 0$ and $g_m$ is an isomorphism. Hence $X\cong_{\mathcal{T}/\mathcal{T}^{p}} X_{\leq m}$. Since $\xi_m\circ g_{m} =g_{m-1}$, while $g_{m-1}$ and $g_{m}$ become isomorphisms in $\mathcal{T}/\mathcal{T}^{p}$, so does $\xi_m$.
\end{proof}

\begin{remark}\label{remark_psi}
Given $X$ and $Y$ in $\mathcal{T}$, applying the functor $\mathcal{T}(Y,-)$ to a direct system as in Lemma \ref{lemma_directsystem}, we obtain a direct system of the form
\begin{align*}
\xymatrix@C=2.5em{
    \mathcal{T}(Y,X)\ar[r]&\mathcal{T}(Y,X_{\leq 0})\ar[r] &\mathcal{T}(Y,X_{\leq -1})\ar[r]& \mathcal{T}(Y,X_{\leq -2})\ar[r] &\cdots.
    }
\end{align*}
Moreover, passing to the quotient category $\mathcal{T}/\mathcal{T}^p$ using the quotient functor $Q: \mathcal{T} \rightarrow \mathcal{T} /\mathcal{T}^p$, we obtain another direct system and a commutative diagram of the form
\begin{align*}
\xymatrix@C=2.5em{
    \mathcal{T}(Y,X)\ar[r]\ar[d]^{Q(-)}&\mathcal{T}(Y,X_{\leq 0})\ar[r]\ar[d]^{Q(-)}&\mathcal{T}(Y,X_{\leq -1})\ar[r]\ar[d]^{Q(-)}&\mathcal{T}(Y,X_{\leq -2})\ar[r]\ar[d]^{Q(-)}&\cdots\\
    \mathcal{T}/\mathcal{T}^{p}(Y,X)\ar[r]^{\sim}&\mathcal{T}/\mathcal{T}^{p}(Y,X_{\leq 0})\ar[r]^{\sim}&\mathcal{T}/\mathcal{T}^{p}(Y,X_{\leq -1})\ar[r]^{\sim}&\mathcal{T}/\mathcal{T}^{p}(Y,X_{\leq -2})\ar[r]^-{\sim}&\cdots,
    }
\end{align*}
where all the arrows in the bottom row are isomorphisms by Lemma \ref{lemma_iso_quotient_cat}.
Then, by the universal property of direct systems, there exists a unique morphism of the form
\begin{align*}
    \Psi: \underset{p}{\varinjlim}\,\mathcal{T}(Y, X_{\leq -p})\rightarrow \underset{p}{\varinjlim}\,\mathcal{T}/\mathcal{T}^{p}(Y,X_{\leq -p})
\end{align*}
such that the diagram
\begin{align}\label{diagram_psi_lim}
    \xymatrix{
    \mathcal{T}(Y,X_{\leq -q})\ar[d]_{\nu_{-q}}\ar[r]^{Q(-)}&\mathcal{T}/\mathcal{T}^{p}(Y,X_{\leq -q})\ar[d]^-{\rotatebox{90}{$\sim$}}_{\overline{\nu}_{-q}}\\
    \underset{p}{\varinjlim}\,\mathcal{T}(Y,X_{\leq -p})\ar@{-->}[r]^-{\Psi}&\underset{p}{\varinjlim}\,\mathcal{T}/\mathcal{T}^{p}(Y,X_{\leq -p})
    }
\end{align}
commutes for every $q\geq 0$.
\end{remark}

\begin{lemma}\label{lemma_psi_iso}
For a construction as in Remark \ref{remark_psi}, the  morphism
\begin{align*}
    \Psi: \underset{p}{\varinjlim}\,\mathcal{T}(Y, X_{\leq -p})\rightarrow \underset{p}{\varinjlim}\,\mathcal{T}/\mathcal{T}^{p}(Y,X_{\leq -p})
\end{align*}
is an isomorphism.
\end{lemma}

\begin{proof}
We first prove that $\Psi$ is injective. Suppose that $\beta\in\varinjlim\,\mathcal{T}(Y, X_{\leq -p})$ is such that $\Psi(\beta)=0$. By \cite[Lemma 5.30(i)]{Rot}, we have that $\beta$ comes from an element in one of the Hom-spaces of the direct system, say from the element $\alpha\in \mathcal{T}(Y, X_{\leq -q})$. Then, considering the commutative diagram (\ref{diagram_psi_lim}), we have
\begin{align*}
\overline{\nu}_{-q}\circ Q(\alpha)=\Psi\circ \nu_{-q}(\alpha)=\Psi(\beta)=0.
\end{align*}
Since $\overline{\nu}_{-q}$ is an isomorphism, we have that $Q(\alpha)=0$. Hence there exists a morphism $f: X_{\leq -q}\rightarrow K$ in the multiplicative system being inverted when we pass to $\mathcal{T}/\mathcal{T}^p$ such that $f\circ \alpha=0$. Consider the triangle extending $f$, say
\begin{align*}
\xymatrix{
Z\ar[r]^\gamma& X_{\leq -q}\ar[r]^f& K\ar[r]&\Sigma Z,
}
\end{align*}
where $Z\in\mathcal{T}^p$. By \cite[Lemma 4.9]{JH}, since $Z\in\mathcal{T}^p$, we have that $\mathcal{T}(Z,\Sigma^i \mathcal{S})\neq 0$ for finitely many $i\in\mathbb{Z}$. Hence there is an integer $j$ such that $\mathcal{T}(Z,\Sigma^{>j}\mathcal{S})=0$, that is $Z\in \mathcal{T}_{\geq -j}={}^\perp(\mathcal{T}_{\leq -j-1})$. Then, as $X_{\leq -j-1}\in\mathcal{T}_{\leq -j-1}$, we have that $\mathcal{T}(Z,X_{\leq -j-1})=0$. Since $ f\circ\alpha= 0$, there exists a morphism $c: Y\rightarrow Z$ such that $\gamma\circ c=\alpha$, that is such that the following commutes:
\begin{align*}
    \xymatrix@C=5em{
    & Y\ar[d]^\alpha\ar[rd]^0\ar@{-->}[ld]_c\\
    Z\ar[r]^\gamma &X_{\leq -q}\ar[r]^f&K\ar[r]&\Sigma Z.
    }
\end{align*}
If $-j-1>-q$, then $Z\in\mathcal{T}_{\geq -j}\subseteq \mathcal{T}_{\geq -q+1}={}^\perp(\mathcal{T}_{\leq -q})$. So $\gamma=0$ and $\alpha=\gamma\circ c=0$ and $\beta=\nu_{-q}(\alpha)=0$. In the other case, that is $-j-1\leq -q$, the direct system gives us a morphism $\xi: X_{\leq -q}\rightarrow X_{\leq -j-1}$.
Then, as $\mathcal{T}(Z,X_{\leq -j-1})=0$, we have that $\xi\circ \gamma=0$ and so $\xi\circ\alpha=\xi\circ \gamma\circ c=0$. Consider the commutative diagram
\begin{align*}
    \xymatrix@C=3em{
    \mathcal{T}(Y,X_{\leq -q})\ar[r]^{\xi_*}\ar[rd]_{\nu_{-q}}& \mathcal{T}(Y,X_{\leq -j-1})\ar[d]^{\nu_{-j-1}}\\
    & \underset{p}{\varinjlim}\,\mathcal{T}(Y,X_{\leq -p}).
    }
\end{align*}
We have that $0=\nu_{-j-1}(0)=\nu_{-j-1}(\xi\circ \alpha)=\nu_{-j-1}\circ\xi_*(\alpha)=\nu_{-q}(\alpha)=\beta$. Hence $\Psi$ is injective.

It remains to show that $\Psi$ is also surjective. Let $\gamma$ be an element in ${\varinjlim}\,\mathcal{T}/\mathcal{T}^{p}(Y,X_{\leq -p})$. By \cite[Lemma 5.30(i)]{Rot}, $\gamma$ comes from an element in one of the Hom-spaces of the direct system and since these are all isomorphic, we may assume $\gamma$ comes from an element of the form
\begin{align*}
    \alpha=\left[ \vcenter{\xymatrix@R=1em @C=1em{
& Z&
\\ Y\ar[ru]^f& &X\ar[lu]_{s}}}\right]\in\mathcal{T}/\mathcal{T}^{p}(Y,X),
\end{align*}
where the triangle extending $s$, say
\begin{align*}
    \xymatrix@C=2.5em{
    W\ar[r]&X\ar[r]_{\sim}^s&Z\ar[r]&\Sigma W,
    }
\end{align*}
has $W$ in $\mathcal{T}^{p}$ since $s$ is in the multiplicative system being inverted. Then, \cite[Lemma 4.9]{JH} implies that there exists an integer $j$ such that $\mathcal{T}(W, \Sigma^{>j}\mathcal{S})=0$, that is $W\in\mathcal{T}_{\geq -j}={}^\perp (\mathcal{T}_{\leq -j-1})$.
As $X_{\leq -j-1}\in \mathcal{T}_{\leq -j-1}$, we have that $\mathcal{T}(W,X_{\leq -j-1})=0$.
Without loss of generality, we may assume $j$ is positive. Then, the direct system from Lemma \ref{lemma_directsystem} gives a morphism $\xi: X\rightarrow X_{\leq -j-1}$ and we have a commutative diagram:
\begin{align*}
\xymatrix@C=5em{
W\ar[r]\ar[rd]_0&X\ar[r]^s\ar[d]^\xi&Z\ar[r]\ar@{-->}[ld]^z&\Sigma W.\\
&X_{\leq -j-1}&&
}
\end{align*}
Then, we have that
\begin{align*}
\left[ \vcenter{\xymatrix@R=1em @C=1em{
& {X_{\leq -j-1}}&
\\ X\ar[ru]^-{\xi}& &{X_{\leq -j-1}}\ar@{=}[lu]}}\right]\circ \left[ \vcenter{\xymatrix@R=1em @C=1em{
& Z&
\\ Y\ar[ru]^f& &X\ar[lu]_{s}}}\right]
=\left[ \vcenter{\xymatrix@R=1em @C=1em{
& {X_{\leq -j-1}}&
\\ Y\ar[ru]^-{z\circ f}& &{X_{\leq -j-1}}\ar@{=}[lu]}}\right]= Q(z\circ f).
\end{align*}
Consider the commutative diagram
\begin{align*}
    \xymatrix{
    & \mathcal{T}/\mathcal{T}^{p}(Y,X)\ar[d]^-{\rotatebox{90}{$\sim$}}_{\overline{\mu}}\\
    \mathcal{T}(Y,X_{\leq -j-1})\ar[d]_{\nu_{-j-1}}\ar[r]^{Q(-)}&\mathcal{T}/\mathcal{T}^{p}(Y,X_{\leq -j-1})\ar[d]^-{\rotatebox{90}{$\sim$}}_{\overline{\nu}_{-j-1}}\\
    \underset{p}{\varinjlim}\,\mathcal{T}(Y,X_{\leq -p})\ar@{-->}[r]^-{\Psi}&\underset{p}{\varinjlim}\,\mathcal{T}/\mathcal{T}^{p}(Y,X_{\leq -p})
    }
\end{align*}
Then, we have
\begin{align*}
\gamma=\overline{\nu}_{-j-1}\circ \overline{\mu}(\alpha)=\overline{\nu}_{-j-1}\circ Q(z\circ f)=\Psi\circ \nu_{-j-1}(z\circ f)
\end{align*}
and so $\Psi$ is surjective.
\end{proof}

\begin{notation}
Given $X$ in $\mathcal{T}$, consider a direct system as in Lemma \ref{lemma_directsystem}. Then, for $i\geq 1$, the triangle in $\mathcal{T}$ extending $\xi_{-i}$ is
\begin{align*}
\xymatrix{
X_{\leq -i}\ar[r]^-{\xi_{-i}}& X_{\leq -i-1}\ar[r]&\Sigma C^X_{-i}\ar[r]&\Sigma X_{\leq -i}.
}
\end{align*}
Note that the object $C^X_{-i}$ depends not only on the object $X$, but also on the choice of the truncation triangles from Remark \ref{remark_cot_structure}.
\end{notation}

\begin{lemma}\label{lemma_Cp}
Given $X\in\mathcal{T}$ and an integer $i\geq 1$, we have that each object $C_{-i}^X$ is in $\mathcal{T}_{-i}:=\mathcal{T}_{\leq -i}\cap\mathcal{T}_{> -i-1}$. In particular, $C_{-i}^X\in\mathcal{T}^{p}$.
\end{lemma}
\begin{proof}
Using Remark \ref{remark_cot_structure} and the octahedral axiom, we have a commutative diagram of triangles in $\mathcal{T}$ of the form
\begin{align*}
   \xymatrix@C=3em{
    X_{> -i}\ar[r]^{f_{-i}}\ar[d]^{\alpha_{-i}}& X\ar[r]^{g_{-i}}\ar@{=}[d]& X_{\leq -i}\ar[r]\ar[d]^{\xi_{-i}}&\Sigma X_{> -i}\ar[d]\\
    X_{> -i-1}\ar[r]^{f_{-i-1}}\ar[d]& X\ar[r]^{g_{-i-1}}\ar[d]& X_{\leq -i-1}\ar[r]\ar[d]&\Sigma X_{> -i-1}\ar[d]\\
    C_{-i}^X\ar[r]\ar[d]&0\ar[r]\ar[d]& \Sigma C_{-i}^X\ar@{=}[r]\ar[d]&\Sigma C_{-i}^X\ar[d]\\
    \Sigma X_{> -i}\ar[r]&\Sigma X\ar[r]& \Sigma X_{\leq -i}\ar[r]&\Sigma^2 X_{> -i}.
    }
\end{align*}
Note that
\begin{align*}
X_{>-i-1}\in \mathcal{T}_{>-i-1}
\text{ and } \Sigma X_{>-i}\in \Sigma \mathcal{T}_{>-i}= \mathcal{T}_{>-i-1} .
\end{align*}
Since $\mathcal{T}_{>-i-1}$ is closed under extensions, we have that $C^X_{-i}\in \mathcal{T}_{>-i-1}\subseteq \mathcal{T}^p$. Moreover, note that
\begin{align*}
X_{\leq -i-1}\in \mathcal{T}_{\leq -i-1} \text{ and } \Sigma X_{\leq -i}\in \Sigma\mathcal{T}_{\leq -i}=\mathcal{T}_{\leq -i-1}.
\end{align*}
Since $\mathcal{T}_{\leq -i-1}$ is closed under extensions, we have that $\Sigma C^X_{-i}\in \mathcal{T}_{\leq -i-1}=\Sigma\mathcal{T}_{\leq -i}$ and so $C^X_{-i}\in\mathcal{T}_{\leq -i}$.
\end{proof}

\begin{lemma}\label{lemma_pbig}
Let $X$ and $Y$ be objects in $\mathcal{T}$, then for each possible $C_{-p}^X$, we have that $\mathcal{T}(C_{-p}^X,Y)=0$ for $p\gg 0$.
\end{lemma}
\begin{proof}
By Remark \ref{remark_cot_structure}, $\mathcal{T}= \mathcal{T}^{\leq j}\perp \mathcal{T}^{>j}$ is a bounded $t$-structure for any integer $j$ and so there is an integer $i$ such that $Y\in\mathcal{T}^{>i}$. Pick an integer $p\geq -i$, so that $-p\leq i$. Then, by Lemma \ref{lemma_Cp} we have that
\begin{align*}
C^X_{-p}\in\mathcal{T}_{\leq -p}\subseteq \mathcal{T}_{\leq i}=\mathcal{T}^{\leq i},
\end{align*}
where the last equality holds by Remark \ref{remark_cot_structure}. As $\mathcal{T}= (\mathcal{T}^{\leq i},\mathcal{T}^{>i})$ is a torsion pair, we then have that $\mathcal{T}(C^{X}_{-p}, Y)=0$.
\end{proof}

\begin{lemma}\label{lemma_X_thick}
Let $X$ and $Y$ be objects in $\mathcal{T}$ and let $q\gg 0$ be an integer. We have that $\mathcal{T}(X_{\leq -q}, Y)=0$ for each choice of $X_{\leq -q}$ from the construction in Remark \ref{remark_cot_structure}.
\end{lemma}
\begin{proof}
By Remark \ref{remark_cot_structure}, $\mathcal{T}= \mathcal{T}^{\leq j}\perp \mathcal{T}^{>j}$ is a bounded $t$-structure for any integer $j$ and so there is an integer $i$ such that $Y\in\mathcal{T}^{>i}$.  Pick an integer $q\geq -i$, so that $-q\leq i$. Then, 
\begin{align*}
X_{\leq -q}\in \mathcal{T}_{\leq -q}\subseteq \mathcal{T}_{\leq i}=\mathcal{T}^{\leq i},
\end{align*}
 where the last equality holds by Remark \ref{remark_cot_structure}. 
As $\mathcal{T}= (\mathcal{T}^{\leq i},\mathcal{T}^{>i})$ is a torsion pair, we then have that $\mathcal{T}(X_{\leq -q}, Y)=0$.
\end{proof}

\begin{theorem}\label{thm_directstab}
Let $X$ and $Y$ be objects in $\mathcal{T}$.
\begin{enumerate}[label=(\alph*)]
\item For $p\gg0$, each direct system
\begin{align*}
\xymatrix@C=3em{
    \mathcal{T}(Y,X)\ar[r]&\mathcal{T}(Y,X_{\leq 0})\ar[r]&\mathcal{T}(Y,X_{\leq -1})\ar[r]&\cdots\ar[r]&\mathcal{T}(Y,X_{\leq -p})\ar[r]&\cdots
    }
\end{align*}
stabilizes. Moreover, we have that
    $\underset{p}{\varinjlim}\,\mathcal{T}(Y,X_{\leq -p})\cong \mathcal{T}/\mathcal{T}^{p}(Y, X)$.
\item Fixing truncation triangles for X and Y, we have that
    \begin{align*}
       \underset{q}{\varprojlim} \Big(\underset{p}{\varinjlim}\,\mathcal{T}(Y_{\leq -q},\Sigma X_{> -p})\Big)\cong \mathcal{T}/\mathcal{T}^{p}(Y,X).
    \end{align*}
\end{enumerate}
\end{theorem}

\begin{proof}
(a) Consider any triangle $\xymatrix{C^X_{-i}\ar[r]& X_{\leq -i}\ar[r]^-{\xi_{-i}}& X_{\leq -i-1}\ar[r]& \Sigma C^X_{-i}}$. Applying $\mathcal{T}(Y,-)$ to it, we obtain the exact sequence
\begin{align*}
\xymatrix{
\mathcal{T}(Y, C^X_{-i})\ar[r]& \mathcal{T}(Y, X_{\leq -i})\ar[r]&\mathcal{T}(Y, X_{\leq -i-1})\ar[r]&\mathcal{T}(Y, \Sigma C^X_{-i}).
}
\end{align*}
By Lemma \ref{lemma_Cp}, we have that $C_{-i}^X\in\mathcal{T}^{p}$, so that $\mathcal{T}(Y,C_{-i}^X)\cong D\mathcal{T}(C_{-i}^X, \mathbb{S} Y)$ and $\mathcal{T}(Y,\Sigma C_{-i}^X)\cong D\mathcal{T}(C_{-i}^X, \Sigma^{-1}\mathbb{S}Y)$ by (RS1). For $i\gg 0$, by Lemma \ref{lemma_pbig}, we have that $\mathcal{T}(Y, X_{\leq -i})\cong \mathcal{T}(Y, X_{\leq -i-1})$. Hence the direct system stabilizes as claimed and for $i\gg0$ we have that
\begin{align*}
\underset{p}{\varinjlim}\,\mathcal{T}(Y, X_{\leq -p})\cong \mathcal{T}(Y,X_{\leq -i}).
\end{align*}
Then we have
\begin{align*}
    \underset{p}{\varinjlim}\,\mathcal{T}(Y,X_{\leq -p})\cong \underset{p}{\varinjlim}\,\mathcal{T}/\mathcal{T}^{p}(Y,X_{\leq -p})\cong 
    \mathcal{T}/\mathcal{T}^{p}(Y,X),
\end{align*}
where the first isomorphism holds by Lemma \ref{lemma_psi_iso} and the second one holds because all the Hom-spaces in the direct system over $\mathcal{T}/\mathcal{T}^p$ are isomorphic and isomorphic to the direct limit.

(b) Consider a truncation triangle $\xymatrix{X_{>-p}\ar[r]& X\ar[r]& X_{\leq -p}\ar[r]& \Sigma X_{>-p}}$. Applying $\mathcal{T}(Y_{\leq -q},-)$ to it, we obtain the exact sequence
\begin{align}\label{diagram_proof1}
\xymatrix{
\mathcal{T}(Y_{\leq -q}, X)\ar[r]&\mathcal{T}(Y_{\leq -q}, X_{\leq -p})\ar[r]& \mathcal{T}(Y_{\leq -q}, \Sigma X_{>-p})\ar[r]& \mathcal{T}(Y_{\leq -q}, \Sigma X).
}
\end{align}
Taking the direct limit of (\ref{diagram_proof1}) with respect to $p$, we obtain the exact sequence
\begin{align}\label{diagram_proof2}
\xymatrix{
\mathcal{T}(Y_{\leq -q}, X)\ar[r]&\mathcal{T}/\mathcal{T}^p(Y, X)\ar[r]&\underset{p}{\varinjlim}\, \mathcal{T}(Y_{\leq -q}, \Sigma X_{>-p})\ar[r]& \mathcal{T}(Y_{\leq -q}, \Sigma X).
}
\end{align}
In the above, the first and last terms are unchanged since they do not depend on $p$. Moreover, for the second term, we observed that
\begin{align*}
\underset{p}{\varinjlim}\,\mathcal{T}(Y_{\leq -q}, X_{\leq -p})\cong \mathcal{T}/\mathcal{T}^p (Y_{\leq -q}, X)\cong \mathcal{T}/\mathcal{T}^p (Y, X),
\end{align*}
where we used part (a) for the first isomorphism and Lemma \ref{lemma_iso_quotient_cat}. For $q\gg 0$, by Lemma \ref{lemma_X_thick}, we have that
\begin{align*}
\mathcal{T}(Y_{\leq -q},X)=0 \text{ and } \mathcal{T}(Y_{\leq -q},\Sigma X)=0.
\end{align*}
Hence, for $q\gg 0$, (\ref{diagram_proof2}) becomes
\begin{align*}
\xymatrix{
0\ar[r]&\mathcal{T}/\mathcal{T}^p(Y, X)\ar[r]^-{\sim}&\underset{p}{\varinjlim}\, \mathcal{T}(Y_{\leq -q}, \Sigma X_{>-p})\ar[r]& 0.
}
\end{align*}
By \cite[Exercise 5.22(ii)]{Rot}, chopping off the tail of an inverse system, we obtain the same inverse limit. Hence, we have that
\begin{align*}
\underset{q}{\varprojlim} \Big(\underset{p}{\varinjlim}\,\mathcal{T}(Y_{\leq -q},\Sigma X_{> -p})\Big)\cong \mathcal{T}/\mathcal{T}^{p}(Y,X).
\end{align*}
\end{proof}

\section{Applications of our Theorem}\label{section_application}
In this section, we work in the following more specific setup.
\begin{setup}
Let $(\mathcal{T}, \mathcal{T}^p,\mathcal{S})$ be a $(-d)$-CY triple.
\end{setup}

When $Y$ is a shift of an object in $\mathcal{S}$, we can say more about when the direct system in Theorem \ref{thm_directstab}(a) stabilizes and we obtain the following result.

\begin{corollary}\label{coro_YisS}
Let $X\in \mathcal{T}$ and $Y\in\mathcal{S}$. Then, for any integer $j$ and any choice of truncation triangles for $X$, we have that
\begin{align*}
\mathcal{T}/\mathcal{T}^p(\Sigma^j Y, X)\cong \mathcal{T}(\Sigma^j Y, X_{\leq -j+d-1}).
\end{align*}
\end{corollary}

\begin{proof}
By Theorem \ref{thm_directstab}(a), for $p\gg 0$, we have that $\mathcal{T}(\Sigma^j Y,X_{\leq -p})\cong \mathcal{T}/\mathcal{T}^{p}(\Sigma^j Y, X)$. We show that $p=j-d+1$ is big enough. For any integer $l$, consider a truncation triangle
\begin{align*}
\xymatrix{
X_{\leq -l}\ar[r]^{\xi_{-l}}& X_{\leq -l-1}\ar[r]& \Sigma C^X_{-l}\ar[r]& \Sigma X_{\leq -l},
}
\end{align*}
where $C^X_{-l}\in\mathcal{T}_{-l}$ by Lemma \ref{lemma_Cp}.
Applying $\mathcal{T}(\Sigma^j Y,-)$ to this triangle, we obtain the exact sequence
\begin{align*}
\xymatrix{
\mathcal{T}(\Sigma^j Y, C^X_{-l})\ar[r]&
\mathcal{T}(\Sigma^j Y, X_{\leq -l})\ar[r]^\alpha&
\mathcal{T}(\Sigma^j Y, X_{\leq -l-1})\ar[r]&
\mathcal{T}(\Sigma^j Y, \Sigma C^X_{-l}).
}
\end{align*}
Since $C^X_{-l}\in\mathcal{T}^p$, we have that
\begin{align*}
\mathcal{T}(\Sigma^j Y, C^X_{-l})\cong D\mathcal{T}(C^X_{-l}, \Sigma^{j-d} Y) \text{ and }
\mathcal{T}(\Sigma^j Y, \Sigma C^X_{-l})\cong D\mathcal{T}(C^X_{-l}, \Sigma^{j-d-1} Y).
\end{align*}
Moreover, since $C^{X}_{-l}\in\mathcal{T}_{\leq -l}$, we have that $\mathcal{T}(C^X_{-l}, \Sigma^{<l} \mathcal{S})=0$ and so $\mathcal{T}(C^X_{-l}, \Sigma^{<l} Y)=0$. Hence, if $l\geq j-d+1$, the morphism $\alpha$ is an isomorphism.
\end{proof}

Moreover, if we also fix $X$ to be an object in $\mathcal{S}$, then Corollary \ref{coro_YisS} has the following important special case.
\begin{corollary}\label{coro_prop(2)}
Let $X$ and $Y$ be objects in $
\mathcal{S}$. Then, we have that $\mathcal{T}/\mathcal{T}^p(\Sigma^j Y, X)=0$ for $1\leq j\leq d$.
\end{corollary}

\begin{proof}
By Corollary \ref{coro_YisS}, we have that $\mathcal{T}/\mathcal{T}^p(\Sigma^j Y, X)\cong \mathcal{T}(\Sigma^j Y, X_{\leq -j+d-1})$. Consider a truncation triangle
\begin{align*}
\xymatrix{
X_{> -j+d-1}\ar[r]^-{f}& X\ar[r]^-{g}&X_{\leq -j+d-1}\ar[r]&\Sigma X_{> -j+d-1},
}
\end{align*}
where we may assume $g$ is left minimal, see Remark \ref{remark_cot_structure}, and so $f$ is right minimal by \cite[Lemma 2.5]{KH}. If $1\leq j\leq d-1$, then $-j+d-1\geq 0$ and $X_{>-j+d-1}\in\mathcal{T}_{>-j+d-1}\subseteq \mathcal{T}_{>0}$. Since $\mathcal{S}\subseteq\mathcal{T}_{\leq 0}$, $X\in \mathcal{S}$ and $(\mathcal{T}_{> 0},\mathcal{T}_{\leq 0})$ is a torsion pair, we have that $\mathcal{T}(X_{> -j+d-1}, X)=0$. Then $f=0$ and, as it is right minimal, we have that $X_{> -j+d-1}=0$ and $X\cong X_{\leq -j+d-1}$. Hence
\begin{align*}
\mathcal{T}/\mathcal{T}^p(\Sigma^j Y, X)\cong \mathcal{T}(\Sigma^j Y, X_{\leq -j+d-1})\cong \mathcal{T}(\Sigma^j Y, X)\cong \mathcal{T}( Y,\Sigma^{-j} X)=0
\end{align*}
by (RS2), since $-j\leq -1$.

It remains to show that the result holds for the case $j=d$. Note that in each truncation triangle
\begin{align*}
\xymatrix{
X_{> -1}\ar[r]^-{f}& X\ar[r]^-{g}&X_{\leq -1}\ar[r]&\Sigma X_{> -1}
},
\end{align*}
we have that $X$ and $\Sigma^{-1}X_{\leq -1}$ are in $\mathcal{T}_{\leq 0}$. Then, as $\mathcal{T}_{\leq 0}$ is closed under extensions, we have that $X_{>-1}\in\mathcal{T}_{\leq 0}$. Moreover, such a truncation triangle induces the exact sequence
\begin{align*}
\xymatrix{
\mathcal{T}(\Sigma^d Y, X)\ar[r]&
\mathcal{T}(\Sigma^d Y, X_{\leq -1})\ar[r]&
\mathcal{T}(\Sigma^d Y, \Sigma X_{>-1}).
}
\end{align*}
Note that as $Y\in\mathcal{T}_{\leq 0}$, we have that the first term is zero. Moreover, as $\Sigma X_{>-1}\in \mathcal{T}^p$, by (RS1) we have that
\begin{align*}
\mathcal{T}(\Sigma^d Y, \Sigma X_{>-1})\cong D\mathcal{T}(\Sigma^{d+1}X_{>-1}, \Sigma^d Y)\cong D\mathcal{T}(X_{>-1}, \Sigma^{-1} Y)=0,
\end{align*}
where the last equality holds as $X_{>-1}\in\mathcal{T}_{\leq 0}$ by above and $\Sigma^{-1}Y\in\Sigma^{-1}\mathcal{S}$. Hence, we have that $\mathcal{T}(\Sigma^d Y, X_{\leq -1})=0$ and so $\mathcal{T}(\Sigma^d Y, X)=0$ by Corollary \ref{coro_YisS}.
\end{proof}

\begin{lemma}\label{lemma_prop(1)}
We have that dim$\,{}_k \mathcal{T}/\mathcal{T}^p(S_i,S_j)=\delta_{S_iS_j}$ for every $S_i$ and $S_j$ in $\mathcal{S}$.
\end{lemma}

\begin{proof}
By Corollary \ref{coro_YisS}, we have that $\mathcal{T}/\mathcal{T}^p( S_i, S_j)\cong \mathcal{T}( S_i, (S_j)_{\leq d-1}).$
Note that since $\mathcal{S}\subseteq\mathcal{T}_{\leq 0}$, then $S_j\in\mathcal{T}_{\leq 0}$. Consider a truncation triangle
\begin{align*}
\xymatrix{
(S_j)_{>d-1}\ar[r]^-f& S_j\ar[r]&(S_j)_{\leq d-1}\ar[r]& \Sigma (S_j)_{>d-1},
}
\end{align*}
where we may assume $f$ is right minimal.
Note that $(S_j)_{>d-1}\in\mathcal{T}_{>d-1}\subseteq \mathcal{T}_{>0}$ and as $(\mathcal{T}_{>0}, \mathcal{T}_{\leq 0})$ is a torsion pair, we have that $f=0$. As $f$ is right minimal, then $(S_j)_{>d-1}=0$ and $S_j\cong (S_j)_{\leq d-1}$. Then
\begin{align*} 
\mathcal{T}/\mathcal{T}^p(S_i, S_j)\cong \mathcal{T}(S_i, (S_j)_{\leq d-1})\cong \mathcal{T}(S_i, S_j).
\end{align*}
The result then follows because $\mathcal{S}$ is SMC in $\mathcal{T}$ by (RS2).
\end{proof}
\begin{theorem}\label{lemma_negCY}
The category $\mathcal{T}/\mathcal{T}^p$ is Hom-finite and $(-d-1)$-Calabi-Yau.
\end{theorem}
\begin{proof}
Let $X$ and $Y$ be in $\mathcal{T}$. By Theorem \ref{thm_directstab}(a), for $p\gg 0$ we have that
\begin{align*}
\mathcal{T}/\mathcal{T}^{p}(Y, X)\cong \mathcal{T}(Y,X_{\leq -p})
\end{align*}
and $\mathcal{T}$ is Hom-finite by assumption, see (RS0). Hence $\mathcal{T}/\mathcal{T}^p$ is Hom-finite.

Now let $q$, $p$ be integers. For any choice of $Y_{\leq -q}$, applying the covariant functor $D\mathcal{T}(-, Y_{\leq -q})$ to a shift of a truncation triangle
\begin{align*}
    X_{>-p}\rightarrow X\rightarrow X_{\leq -p}\rightarrow \Sigma X_{>-p},
\end{align*}
we obtain the exact sequence
\begin{align*}
\xymatrix{
D\mathcal{T}(\Sigma^d X_{\leq -p}, Y_{\leq -q})\ar[r]&
D\mathcal{T}(\Sigma^{d+1} X_{> -p}, Y_{\leq -q})\ar[r]&
D\mathcal{T}(\Sigma^{d+1} X, Y_{\leq -q})\ar[r]&
D\mathcal{T}(\Sigma^{d+1} X_{\leq -p}, Y_{\leq -q}).
}
\end{align*}
Taking the direct limit of the above with respect to $p$, we obtain the exact sequence
\begin{align*}
\xymatrix{
0\ar[r]&
\underset{p}{\varinjlim}\, D\mathcal{T}(\Sigma^{d+1} X_{> -p}, Y_{\leq -q})\ar[r]&
D\mathcal{T}(\Sigma^{d+1} X, Y_{\leq -q})\ar[r]&
0,
}
\end{align*}
where
\begin{align*}
\underset{p}{\varinjlim}\,D\mathcal{T}(\Sigma^d X_{\leq -p}, Y_{\leq -q})\cong \underset{p}{\varinjlim}\,D\mathcal{T}( X_{\leq -p}, \Sigma^{-d}Y_{\leq -q})=0
\end{align*}
by Lemma \ref{lemma_X_thick} and similarly the last term is zero. Hence
\begin{align}\label{diagram_proof_CY}
D\mathcal{T}(\Sigma^{d+1} X_{> -p}, Y_{\leq -q})\cong
D\mathcal{T}(\Sigma^{d+1} X, Y_{\leq -q}).
\end{align}
Moreover, by Theorem \ref{thm_directstab}(a) the inverse system
\begin{align*}
\xymatrix{
D\mathcal{T}(\Sigma^{d+1}X, Y)&
D\mathcal{T}(\Sigma^{d+1}X, Y_{\leq 0})\ar[l]&
D\mathcal{T}(\Sigma^{d+1}X, Y_{\leq -1})\ar[l]&\cdots\ar[l]
}
\end{align*}
stabilizes and so
\begin{align}\label{diagram_proof_CY2}
\underset{q}{\varprojlim} D\mathcal{T}(\Sigma^{d+1}X, Y_{\leq -q})\cong D (\underset{q}{\varinjlim}\mathcal{T}(\Sigma^{d+1}X, Y_{\leq -q})).
\end{align}
Hence, we have that
\begin{align*}
\mathcal{T}/\mathcal{T}^p(Y,X) &\cong \underset{q}{\varprojlim}(\underset{p}{\varinjlim}\mathcal{T}(Y_{\leq -q}, \Sigma X_{>-p}))\\
&\cong\underset{q}{\varprojlim}(\underset{p}{\varinjlim}D\mathcal{T}( \Sigma^{d+1} X_{>-p}, Y_{\leq -q}))\\
&\cong \underset{q}{\varprojlim}(D\mathcal{T}( \Sigma^{d+1} X, Y_{\leq -q}))\\
&\cong D (\underset{q}{\varinjlim}\mathcal{T}(\Sigma^{d+1}X, Y_{\leq -q}))\\
&\cong D\mathcal{T}/\mathcal{T}^p(\Sigma^{d+1}X, Y),
\end{align*}
where the first isomorphism follows by Theorem \ref{thm_directstab}(b), the second by (RS1) with $\mathbb{S}=\Sigma^{-d}$, the third by (\ref{diagram_proof_CY}), the fourth by (\ref{diagram_proof_CY2}) and the last one by Theorem \ref{thm_directstab}(a).
\end{proof}

To prove the following, we apply \cite[Lemma 3.1]{IYa}. In order to be able to do this, we first need to check that the required conditions (T0)-(T2) from \cite[Section 1.2]{IYa} hold. Note that these conditions have also been checked by Jin in the proof of \cite[Theorem 4.5]{JH}. Even if both us and Jin rely on results by Iyama and Yang from \cite{IYa},  our arguments differ as Jin uses the main theorem while we apply another lemma.

\begin{lemma}\label{lemma_prop(3)}
We have that
\begin{align*}
\mathcal{T}/\mathcal{T}^p\cong
\langle \mathcal{S}\rangle * \Sigma^{-1}\langle\mathcal{S}\rangle*\dots* \Sigma^{1-d}\langle\mathcal{S}\rangle* \Sigma^{-d} \langle\mathcal{S}\rangle.
\end{align*}
\end{lemma}

\begin{proof}
We check that the required conditions (T0)-(T2) from \cite[Section 1.2]{IYa} hold:
\begin{enumerate}
\item[(T0)] $\mathcal{T}$ is a triangulated category and $\mathcal{T}^p$ is a thick subcategory,
\item[(T1)] $\mathcal{T}^p$ has a torsion pair $\mathcal{T}^p=\mathcal{X}\perp \mathcal{Y}$,
\item[(T2)] $\mathcal{T}$ has torsion pairs $\mathcal{T}=\mathcal{X}\perp \mathcal{X}^\perp= {}^\perp\mathcal{Y}\perp \mathcal{Y}$.
\end{enumerate}

Condition (T0) holds by construction. Let $\mathcal{X}=\mathcal{T}_{>0}$, $\mathcal{Y}=\mathcal{T}_{\leq 0}\cap \mathcal{T}^p$ and $\mathcal{Z}=\mathcal{X}^{\perp}\cap {}^{\perp}(\Sigma \mathcal{Y})$.
We have that $\mathcal{X}^{\perp}=\mathcal{T}_{>0}^\perp=\mathcal{T}_{\leq 0}$ and so $\mathcal{T}=\mathcal{X}\perp \mathcal{X}^\perp$ is a co-$t$-structure by Remark \ref{remark_cot_structure}. Moreover, for any $X\in\mathcal{T}^p$, consider a truncation triangle of the form
\begin{align*}
\xymatrix{
X_{>0}\ar[r]& X\ar[r]& X_{\leq 0}\ar[r]& \Sigma X_{>0},
}
\end{align*}
where $X_{>0}\in\mathcal{T}_{>0}\subseteq \mathcal{T}^p$ and $X_{\leq 0}\in\mathcal{T}_{\leq 0}$. Since $\mathcal{T}^p$ is closed under extensions, we have that $X_{\leq 0}\in\mathcal{T}^p$. Hence  $\mathcal{T}=\mathcal{X}\perp \mathcal{X}^\perp$ restricts to the co-$t$-structure 
\begin{align*}
\mathcal{T}^p=\mathcal{X}\perp(\mathcal{X}^\perp\cap\mathcal{T}^p)=\mathcal{X}\perp (\mathcal{T}_{\leq 0}\cap \mathcal{T}^p)=\mathcal{X}\perp\mathcal{Y},
\end{align*}
hence (T1) holds. It remains to show the existence of the second torsion pair from (T2). We have that $\mathcal{Y}=\mathcal{T}_{\leq 0}\cap \mathcal{T}^p= {}^{\perp}(\Sigma^{<0}\mathcal{S})\cap \mathcal{T}^p$. Then, for any $Y\in\mathcal{Y}$, by (RS1) we have that
\begin{align*}
D\mathcal{T}(\Sigma^{<0}\mathcal{S}, \mathbb{S}Y)\cong \mathcal{T}(Y,\Sigma^{<0}\mathcal{S})=0.
\end{align*}
Hence $\mathbb{S} Y\in (\Sigma^{<0}\mathcal{S})^{\perp}$ and $\mathcal{Y}\subseteq \mathbb{S}^{-1}(\Sigma^{<0}\mathcal{S})^\perp$. Moreover, we have that $(\Sigma^{<0}\mathcal{S})^\perp\subseteq \mathcal{T}^p$ by (RS2) and so, using (RS1) one can check that $\mathbb{S}^{-1}(\Sigma^{<0}\mathcal{S})^\perp\subseteq \mathcal{T}^p$. Now, for any $X\in\mathbb{S}^{-1}(\Sigma^{<0}\mathcal{S})^\perp$, we have that $\mathbb{S}X\in(\Sigma^{<0}\mathcal{S}^\perp)\subseteq \mathcal{T}^p$ and
\begin{align*}
\mathcal{T}(X,\Sigma^{<0}\mathcal{S})\cong D\mathcal{T}(\Sigma^{<0}\mathcal{S},\mathbb{S}X)=0
\end{align*}
 by (RS1), so that $X\in{}^\perp(\Sigma^{<0}\mathcal{S})$ and $\mathbb{S}^{-1}(\Sigma^{<0}\mathcal{S})^\perp\subseteq{}^\perp(\Sigma^{<0}\mathcal{S})$. Hence $\mathcal{Y}=\mathbb{S}^{-1}(\Sigma^{<0}\mathcal{S})^\perp$. Using the second co-$t$-structure from (RS2), we then conclude that $\mathcal{T}={}^\perp \mathcal{Y}\perp \mathcal{Y}$
is a co-$t$-structure and (T2) holds.

Since conditions (T0)-(T2) are satisfied, we apply \cite[Lemma 3.1]{IYa} and conclude that for any $X$ in $\mathcal{T}$, there is an object $Z$ in $\mathcal{Z}$ such that $X\cong_{\mathcal{T}/\mathcal{T}^p} Z$. Note that, since $\mathbb{S}=\Sigma^{-d}$ by assumption, we have that
\begin{align*}
{}^\perp \mathcal{Y}=\mathbb{S}^{-1}(\Sigma^{\geq 0}\mathcal{S})^{\perp}=\Sigma^d(\Sigma^{\geq 0}\mathcal{S})^{\perp}=\mathcal{T}^{>-d}.
\end{align*}
Recalling that $\mathcal{T}_{\leq 0}=\mathcal{T}^{\leq 0}$ by Remark \ref{remark_cot_structure}, we have that
\begin{align*}
\mathcal{Z}=\mathcal{X}^{\perp}\cap {}^\perp(\Sigma\mathcal{Y})=\mathcal{T}^{\leq 0}\cap \mathcal{T}^{\geq -d}.
\end{align*}
Moreover, by \cite[Proposition 1.3.13]{BBD}, we have that
\begin{align*}
\mathcal{T}^{\leq 0}\cap \mathcal{T}^{\geq -d}=
 \Sigma^d\langle \mathcal{S}\rangle * \Sigma^{d-1}\langle\mathcal{S}\rangle*\dots* \Sigma\langle\mathcal{S}\rangle* \langle\mathcal{S}\rangle.
\end{align*} 
Hence, we conclude that
\begin{align*}
\mathcal{T}/\mathcal{T}^p&\cong\mathcal{T}^{\leq 0}\cap \mathcal{T}^{\geq -d} \\&\cong
\Sigma^d\langle \mathcal{S}\rangle * \Sigma^{d-1}\langle\mathcal{S}\rangle*\dots* \Sigma\langle\mathcal{S}\rangle* \langle\mathcal{S}\rangle\\
&\cong
\langle \mathcal{S}\rangle * \Sigma^{-1}\langle\mathcal{S}\rangle*\dots* \Sigma^{1-d}\langle\mathcal{S}\rangle* \Sigma^{-d} \langle\mathcal{S}\rangle.\\
\end{align*}
\end{proof}

\begin{theorem}\label{coro_SMS}
We have that $\mathcal{S}$ is a $(d+1)$-SMS in $\mathcal{T}/\mathcal{T}^p$.
\end{theorem}
\begin{proof}
The three properties from Definition \ref{defn_SMS} correspond respectively to Lemma \ref{lemma_prop(1)}, Corollary \ref{coro_prop(2)} and Lemma \ref{lemma_prop(3)}.
\end{proof}

\end{document}